\documentclass[leqno, 11pt]{amsart}
\input{article.sty}

\newcommand{\sect}{\mathrm{Sect}}

\title[Sobolev and Michael-Simon inequalities via the ABP method beyond EVG]{Sobolev and Michael-Simon inequalities via the ABP method beyond Euclidean volume growth}

\author[Debora Impera]{Debora Impera}
\address[Debora Impera]{Dipartimento di Scienze Matematiche ``Giuseppe Luigi Lagrange", Politecnico di Torino, Corso Duca degli Abruzzi, 24, Torino, Italy, I-10129}
\email{debora.impera@polito.it}

\author[Michele Rimoldi]{Michele Rimoldi}
\address[Michele Rimoldi]{Dipartimento di Scienze Matematiche ``Giuseppe Luigi Lagrange", Politecnico di Torino, Corso Duca degli Abruzzi, 24, Torino, Italy, I-10129}
\email{michele.rimoldi@polito.it}

\author[Giona Veronelli]{Giona Veronelli}
\address[Giona Veronelli]{Dipartimento di Matematica e Applicazioni, Universit\`a degli Studi di Milano-Bicocca, via R. Cozzi 53, I-20126 Milano, Italy}
\email{giona.veronelli@unimib.it}

\date{\today}

\begin{document}

\thanks{D. Impera and M. Rimoldi are partially supported by INdAM-GNSAGA G. Veronelli is partially supported by INdAM-GNAMPA. D. I. and M. R. acknowledge partial support by the PRIN 2022 project ``Real and Complex Manifolds: Geometry and Holomorphic Dynamics'' - 2022AP8HZ9. G. V. acknowledge that this work was partially supported by the Simons Foundation grant (award no. SFI-MPS-T-Institutes-00010825) and from State Treasury funds as part of a task commissioned by the Minister of Science and Higher Education under the project “Organization of the Simons Semesters at the Banach Center - New Energies in 2026-2028” (agreement no. MNiSW/2025/DAP/491).}



 \begin{abstract}
We develop an ABP approach to Sobolev and Michael–Simon type inequalities under volume noncollapsing assumptions. The main new observation is a refinement of Brendle’s contact-set argument: the ABP image contains the full geodesic ball centered at the minimum point of the Neumann potential, with radius equal to the ABP parameter. This allows one to use lower bounds for the volumes of geodesic balls, either at a fixed scale or under prescribed volume-growth assumptions, rather than positive asymptotic volume ratio. The central application is a Michael-Simon type inequality for immersed submanifolds of ambient manifolds with nonnegative sectional curvature and volume noncollapsing. The resulting inequality contains a lower-order term determined by the noncollapsing scale and applies to submanifolds with controlled mean curvature. In the intrinsic case, the same method gives an ABP proof of Varopoulos’ $L^{1}$-Sobolev inequality with lower-order term, identifying the optimal constant in front of the gradient term, as well as explicit lower bounds for the isoperimetric profile in terms of lower bounds on the volumes of geodesic balls. Further geometric applications include Topping-type diameter estimates for submanifolds involving the $L^{n-1}$-norm of the mean curvature and various heat kernel and spectral estimates in the minimal case.
\end{abstract}

\maketitle

\tableofcontents

\section{Introduction}
The ABP method was introduced around 1960 by A. D. Alexandrov, I. Ya. Bakelman, and C. Pucci to prove a celebrated quantitative maximum principle for solutions to elliptic Dirichlet problems. The same technique, applied to the problem $\Delta u\le f$ in $\Omega\Subset\mathbb R^n$ with Neumann boundary condition
$\partial_\nu u =1$ on $\partial\Omega$, yields a purely dimensional lower bound for the norm $\|f_+\|_{L^n(\Omega)}$. In the late 1990s, X. Cabré, \cite{cab-cat,Cab}, cleverly exploited this observation with the choice $f=|\partial \Omega|/|\Omega|$ to obtain a lower bound on $\| |\partial \Omega|/|\Omega| \|^{n}_{L^n(\Omega)}=|\partial \Omega|^n/|\Omega|^{n-1}$,
that is, on the isoperimetric ratio of $\Omega$. Remarkably, the resulting lower bound is optimal and is attained only by Euclidean balls. This provided a new short and elegant proof of the isoperimetric inequality in $\mathbb R^n$. Moreover, Cabré’s proof via the ABP method, together with the companion proof based on optimal transport proposed by D. Cordero-Erausquin, B. Nazaret, and C. Villani \cite{CENV}, turned out to be particularly well suited for extensions to the Riemannian setting.

Recall that an $n$-dimensional complete Riemannian manifold $(M,g)$ is said to satisfy the Euclidean isoperimetric inequality if
\begin{equation}\label{e:isop M}
|\partial\Omega|^n \ge C  |\Omega|^{n-1}
\end{equation}
for all relatively compact domains $\Omega\subset M$. In this work we will always assume that domains have smooth boundary $\partial \Omega$, without addressing regularity issues. Moreover, with a common abuse of notation, we will write $|A|$ for the measure of $A$, where the measure is either the Riemannian volume measure, if $A$ is $n$-dimensional, or the $(n-1)$-dimensional area measure (i.e. the $(n-1)$-dimensional Hausdorff measure) when $A$ is $(n-1)$-dimensional, as in the case $A=\partial \Omega$.
    A well-known result by H. Federer and W. H. Fleming \cite{FF}, states that the isoperimetric inequality \eqref{e:isop M} is equivalent to the $L^1$-Sobolev inequality
\begin{equation}    \label{e:isop L1} \|\varphi\|_{L^{\frac{n}{n-1}}}\le C^{1/n}\|\nabla\varphi\|_{L^1},\ \forall\, \varphi\in C^\infty_c(M)
\end{equation}
with the same constant $C$. This latter implies the validity of an $L^p$-Sobolev inequality
\begin{equation*}    \label{e:isop Lp} \|\varphi\|_{L^{\frac{np}{n-p}}}\le C(p,n)\|\nabla\varphi\|_{L^p},\ \forall\, \varphi\in C^\infty_c(M),
\end{equation*}
and thus the Sobolev embedding $W^{1,p}\hookrightarrow L^{\frac{np}{n-p}}$ on $M$. Other consequences of the validity of an isoperimetric inequality on a manifold, which show its importance, also include for instance the (at least) Euclidean volume growth of geodesic balls, Faber-Krahn type inequalities for the Dirichlet eigenvalues, Gaussian estimates for the heat kernel; see Section \ref{Sec: ApplMinImm} for some references.

The study of isoperimetric inequalities on Riemannian manifolds is a rich subject with a long history; see, for instance, \cite{Chavel}. Some classes of assumptions which are classically known to imply the validity of a, not necessarily sharp, isoperimetric inequality are the following:
\begin{itemize}

\item minimal submanifolds of $\mathbb R^n$, or more generally of manifolds with non-positive curvature. In this setting, the isoperimetric inequality follows from the Michael-Simon inequality; see \cite{MS,HS}. In these results the constant is not optimal. The optimality of the constant for minimal submanifolds in $\mathbb R^n$ was a longstanding problem, recently resolved in low codimension by S. Brendle, \cite{brendle-jams}; see below for more details.
	\item 	manifolds with $\mathrm{Ric}_g\ge 0$ and positive \textit{Asymptotic Volume Ratio}
		\[
		\mathrm{AVR}:=\lim_{r\to\infty} \frac{|B_r(p)|}{|\mathbb B_r(0)|}=\theta > 0,\ \underbrace{\text{for some}}_{\text{hence any}}\ p\in M.\]	
The first proof of the (non sharp) isoperimetric inequality under these assumptions are due to S.T. Yau  \cite{Yau}, building on a general inequality due to C. B. Croke, \cite{Croke}. A broader class of isoperimetric-type inequalities, in the more general settings of manifolds with the doubling property and supporting an $L^1$-Poincaré inequality, were proved by T. Coulhon and L. Saloff-Coste, \cite{CS-C}. 
\end{itemize}
It is clear that the condition $\mathrm{AVR}>0$ is necessary, as it is implied by \eqref{e:isop L1} and the coarea formula. A simple example showing the necessity of this assumption is given by the cylinders $K\times \mathbb R^{n-q}$ with $K$ a closed $q$-dimensional manifold. Nevertheless, weaker $L^1$-Sobolev inequalities of the form 
\begin{equation}    \label{e:Sob L1} \|\varphi\|_{L^{\frac{n}{n-1}}}\le A\|\nabla\varphi\|_{L^1}+ B\|\varphi\|_{L^1},\ \forall\, \varphi\in C^\infty_c(M)
\end{equation}
hold when $\mathrm{Ric}_g\ge 0$ (and in fact also when $\mathrm{Ric}_g\ge -\kappa^2$, $\kappa\in\mathbb R$) under the weaker volume noncollapsing assumption
\[
\uv(r):=\inf_x |B_r(x)|>0,
\]
\cite{Va}.

As alluded to above, S. Brendle recently gave an important contribution to the theory, by generalizing Cabré's version of the ABP method to Riemannian manifolds. In particular, he proved the validity of an isoperimetric inequality with an optimal constant in the setting of complete manifolds with nonnegative Ricci curvature and  positive $\mathrm{AVR}$ as well as Michael-Simon inequalities of the form
\begin{align*}
A\left(\int_\Sigma h^{\frac{n}{n-1}}\,d\mathrm{vol}_\Sigma\right)^{\frac{n-1}{n}}
\le 
\int_\Sigma \sqrt{|\nabla^\Sigma h|^2+\|H\|^2_{\infty}h^2}\,d\mathrm{vol}_\Sigma+ \int_{\partial \Sigma} h \,d\mathrm{vol}_{n-1}
\end{align*}
 for $n$-dimensional submanifolds $\Sigma$ with mean curvature $H$ into $(n+m)$-dimensional manifolds $M$ of non-negative sectional curvature and positive $\mathrm{AVR}$, which are sharp when $m=1,2$, \cite{brendle-jams,brendle}. It should be noted that related ABP techniques had already been employed on manifolds to address problems of different nature, such as certain curvature integral inequalities \cite{WZ,XZ}. Moreover, Brendle's approach has since been adapted to several geometric settings, including the weighted setting, and different curvature assumptions \cite{Johne-arxiv}, \cite{IRV}, \cite{DLL-unweight}, \cite{MaWu}, \cite{IPRV}.

Let us recall the main ingredients of Brendle's proof for manifolds with $\mathrm{Ric}_g\ge 0$ and $\mathrm{AVR}>0$. Given a domain $\Omega$ and a positive function $h$ on $\Omega$, up to suitably rescaling $h$, the Neumann problem
\[
\begin{cases}
\mathrm{div}(h\,\nabla u)= n\,h^{\frac{n}{n-1}}-|\nabla h|,&\textrm{in }\Omega\\
\partial_\nu u = 1,& \text{on }\partial\Omega.
\end{cases}
\]
admits a smooth solution. Consider the map  $\Phi_r(x)=\exp_x(r\nabla u(x))$. Brendle showed that there exists a contact set $A_r\subset \Omega$ (see \eqref{e:contact}) with the
following properties:
\begin{itemize}
    \item[a)] $\Phi_r(A_r)$ contains the set $\{p\in M\ :\ \Omega\subset B_r(p)\}$, which in turn contains $B_{r-\operatorname{diam}(\Omega)}(q)$ for every $q\in \Omega$;
    \item[b)] On $A_{r}$ $|\det d\Phi_r|\le\left(1+ \frac{\Delta u(\bar x)}{n} r\right)^n$, with $\Delta u(\bar x)\le nh(\bar x)^{\frac{1}{n-1}}$.
\end{itemize}

Combining these two facts with the area formula he obtained that
\[
|B_{r-\operatorname{diam}(\Omega)}(q)|\le \int_{A_r} |\det d\Phi_r|d\mathrm{vol}\le \int_\Omega \left(1+ \frac{nh(x)^{\frac{1}{n-1}}}{n} r\right)^nd\mathrm{vol}
\]
Dividing by $r^{n}$, the argument then works because, when $r$ tends to $\infty$, the dependence on the diameter of $ \Omega$ disappears and the asymptotic volume ratio appears.

In this paper we observe that the estimate on the image of $\Phi_r(A_r)$ can be improved to $\Phi_r(A_r)\supset B_r(x_\ast)$, where $x_\ast$ is the minimum of $u$, the solution to the Neumann problem. While $B_r(x_\ast)$ still depends on $\Omega$ through $x_\ast$, this dependence disappears once one assumes a uniform volume lower bound $\uv(r)>0$. Accordingly, noncollapsing (or suitable growth) of the volume of geodesic balls yields new interesting results, as well as new proofs of classical Sobolev-type inequalities. This novelty is particularly relevant in the case of submanifolds, where the inequalities we obtain were, up to our knowledge, unknown also at a qualitative level.

The first result we present is a refinement of a well-known Sobolev inequality on complete nonnegatively curved manifolds with the volume noncollapsing property.

\begin{theorem}\label{thm:optimal Sob}
        Let $(M^n,g)$ be an $n$-dimensional smooth complete Riemannian manifold without boundary. Suppose that $\mathrm{Ric}_g\ge 0$ and that
    \[
    \uv(r_0):=\inf_{x\in M} |B_{r_0}(x)|>0
    \]
for some $r_0>0$.
Then, for all $h\in C^\infty_c(M)$,
\begin{align}\label{e:SobPot}
\left(\int_M \ h^{\frac{n}{n-1}}d\mathrm{vol}\right)^{\frac{n-1}{n}}
\le A\left(\int_M |\nabla h|\,d\mathrm{vol} + n\,r_0^{-1} \int_M  h\,d\mathrm{vol}\right).
\end{align}
with $A=\frac{r_0}{n\,\uv(r_0)^{\frac{1}{n}}}$. Moreover the constant $A$ is optimal, i.e. it cannot be decreased.
\end{theorem}

N.T. Varopoulos proved this result with a non optimal constant $A$ assuming the weaker curvature assumption $\mathrm{Ric}_g\ge - \kappa^2$ for some $\kappa\ge 0$; see \cite{Va} for the original proof and Section \ref{sec_varopoulos} for a new ABP proof. To the best of our knowledge the validity with the optimal constant had not been observed previously. 
It's worth noticing that the value of the optimal constant changes dramatically if one assumes that the injectivity radius is nonnegative. In this case it is known that the best constant in front of the gradient term of the Sobolev inequality is the Euclidean one; see the very recent \cite{MQ} as well as \cite{MNQ} and \cite{He-AJM} for previous results under stronger assumptions. To the best of our knowledge, Theorem \ref{thm:optimal Sob} is the first result in the spirit of the so called $\mathcal A\mathcal B$--program initiated by E. Hebey and collaborators in the 90's (see \cite{DH-MAMS}) where volume noncollapsing instead of positive injectivity radius is considered.

When the volume of balls $B_r(x)$ increases more than linearly in $r$, but less than $r^n$, it is natural to expect the validity of inequalities which in a sense interpolate between \eqref{e:SobPot} and \eqref{e:isop L1}. Such a result was obtained by T. Coulhon and L. Saloff-Coste in \cite[Theorem 3]{CS-C}, under the more general assumptions of doubling and Poincaré property, instead of nonnegative curvature. In Theorem \ref{thm:CSC} we recover their result via the ABP method when $\mathrm{Ric}_M\ge 0$, showing that for every $C>1$ the
isoperimetric profile $\mathrm I (t)$ of $M$ satisfies  
\begin{equation}\label{e:CSC_intro}
    \mathrm I (t) := \inf\{|\partial \Omega|\ :\ | \Omega|=t\} \ge (C-1)n
\frac{t}{\uv^{-1}(C^n\,n^n\,t)}.\end{equation}
It is worth noticing that the ABP approach gives constants which are explicit and asymptotically sharp, at least in some cases. Indeed, if $\mathrm{AVR}>\theta$, then, in the limit as $C\to \infty$, one recovers Brendle's sharp isoperimetric inequality.
Moreover, we deduce \eqref{e:CSC_intro} as a special case of the more general Sobolev type inequality
\begin{equation*}
(C-1)n\int_\Omega h \,d\mathrm{vol} \le \uv^{-1}\left(\frac{\left(Cn\int_\Omega h \,d\mathrm{vol}\right)^n}{\left(\int_\Omega h^{\frac{n}{n-1}}\,d\mathrm{vol}\right)^{n-1}}\right) \int_\Omega |\nabla h|\,d\mathrm{vol},\qquad h\in C^\infty_c(M),
\end{equation*}
see \eqref{e:main est SC}.
\medskip

As alluded to above, one advantage of the ABP method in proving Sobolev inequalities is its flexibility in the setting of immersed submanifolds with controlled mean curvature. In this context as well, the Euclidean volume growth assumption can be weakened, still leading to meaningful, though necessarily weaker, conclusions.

\begin{theorem}\label{th:Michael-Simon}
Let $M^{n+m}$ be a complete Riemannian manifold without boundary with $\sect^M\ge 0$ and $\uv(r):=\inf_{x\in M}|B_r(x)|>0$ for some (hence any) $r>0$.  Let $\Sigma$ be a compact $n$-dimensional submanifold with (possibly empty) smooth boundary immersed in $M$. Then, for all positive smooth functions $h$ on $\Sigma$
\begin{align}\label{MS}
\frac{n\,\uv(r)^{\frac{1}{n}}}{\omega_{m}^{\frac{1}{n}}r^{1+\frac{m}{n}}}\left(\int_\Sigma h^{\frac{n}{n-1}}\,d\mathrm{vol}_\Sigma\right)^{\frac{n-1}{n}}
&\le 
\int_\Sigma \sqrt{|\nabla^\Sigma h|^2+h^2|H|^2}\,d\mathrm{vol}_\Sigma + nr^{-1}\int_\Sigma h\,d\mathrm{vol}_\Sigma+ \int_{\partial \Sigma} h \,d\mathrm{vol}_{n-1}
\end{align}
with $H$ the mean curvature of $\Sigma$.
\end{theorem}

The validity of Michael-Simon type inequalities permits to deduce interesting geometric consequences. For instance, a direct application gives a uniform lower volume bound for closed submanifolds of a complete ambient manifold with nonnegative sectional curvature and noncollapsing volumes in term of the $L^\infty$-norm of the mean curvature; see Corollary \ref{coro:vol}. Following \cite{Topping}, a more refined argument gives upper bounds for the intrinsic diameter of closed immersed submanifolds in terms of their volume and the $L^{n-1}$-norm of $H$. Namely, under the noncollapsing assumption $\uv(r_0)>0$ we prove that
\begin{equation*}
\ \mathrm{diam}_{int}(\Sigma)\leq C(n, m, \uv(r_{0}), r_{0})\left(|\Sigma|+\int_{\Sigma}|H|^{n-1}d\mathrm{vol}_{\Sigma}\right).
\end{equation*}
see Theorem \ref{th:Topping}. Moreover, under the volume growth condition
$\uv(r)\ge \tilde\alpha r^{m+n-q}$ for all $r>1$ and for some $\tilde\alpha>0$ and $q\in [0,n)$, it holds the multiplicative estimate
\[\mathrm{diam}_{int}(\Sigma)\leq  c(n,m,\tilde\alpha)\max\{1\,;\,|\Sigma|^{\frac{(n-1)q}{n-q}}\}\int_{\Sigma}|H|^{n-1}d\mathrm{vol}_{\Sigma};\]
see Theorem \ref{th:ToppingIntermediate} for more details.
Remarkably, this latter result holds even without the a priori assumption that $\Sigma$ is closed, provided that its intrinsic volumes do not grow too rapidly. This aspect appears to have been previously overlooked in the literature, even in the Euclidean setting investigated by P. Topping.

Finally, it is well-known that the validity of an isoperimetric inequality on a complete manifold has many important consequences. 
In our setting, minimal submanifolds do not satisfy a genuine Euclidean isoperimetric inequality in general. Nevertheless, following \cite{gr-hkub}, analogous conclusions can still be derived from the Sobolev-type inequalities available here. For instance, under a polynomial lower bound on the volume growth of the ambient space, we obtain an isoperimetric inequality in the spirit of Theorem \ref{thm:CSC}, which in turn yields:
\begin{itemize}
    \item polynomial lower bounds for the  volume of intrinsic geodesic balls;
    \item Faber-Krahn-type estimates for the Dirichlet eigenvalues of bounded domains;
    \item heat kernel (and hence Green kernel) upper bounds. 
\end{itemize}
Compared with the usual results under Euclidean volume growth, these estimates are weaker as the exponents reflect the prescribed polynomial control on the ambient volume growth; see Proposition \ref{pr:Grig} and Corollary \ref{coro:Grig}. \\\smallskip

The paper is organized as follows. In Section \ref{sec:ABP} we present our ABP approach in the setting of nonnegatively Ricci curved manifolds with volume noncollapsing. In Section  \ref{sec:Sob} we deduce various Sobolev-type inequalities from the ABP estimate, in particular proving Theorems \ref{thm:optimal Sob} and \ref{thm:CSC}. In Section \ref{sec:MS} we adapt the ABP estimate of Section \ref{sec:ABP} to the case of immersed submanifolds, thus proving the Michael-Simon inequality stated in Theorem \ref{th:Michael-Simon}. The rest of the paper is devoted to applications to the geometry of submanifolds, either minimal (Section \ref{Sec: ApplMinImm}) or with $L^{n-1}$-bounded mean curvature (Section \ref{sec:diam}).


\section{A general ABP estimate}\label{sec:ABP}

We are going to prove the following general estimate, which will be applied to different settings in the subsequent session. 
\begin{proposition}\label{prop:general}
    Let $(M^n,g)$ be an $n$-dimensional smooth complete Riemannian manifold without boundary. Suppose that $\mathrm{Ric}_g\ge - \kappa^2$ for some $\kappa\ge 0$. Suppose also that
    \[
    \uv(r):=\inf_{x\in M} |B_r(x)|>0
    \]
for some (hence any) $r>0$. Then, for any $B\ge 0$, any compact connected domain $\Omega\subset M$ with smooth boundary and any $h$ a positive smooth function on $\Omega$ it holds
\begin{align}\label{e:general}
 \uv(r)\le \begin{cases}\int_\Omega \left[ r^{-1} +\tilde h^{\frac{1}{n-1}}-B\right]_{+}^n r^n d\mathrm{vol},&\text{if }\kappa=0,\\
\int_\Omega \left[\cosh\left(\frac{\kappa\, r}{\sqrt{n}}\right) r^{-1} +   \frac{\sqrt{n}\,  \left(\tilde h^{\frac{1}{n-1}}-B\right)}{\kappa\,r} \sinh\left(\frac{\kappa\, r}{\sqrt{n}}\right)\right]_{+}^n r^n\,  d\mathrm{vol},&\text{if }\kappa>0,
\end{cases}
\end{align}
where \[
\tilde h(x)=h(x)\left(\frac{\int_\Omega |\nabla h|\,d\mathrm{vol} + B\,n \int_\Omega h\,d\mathrm{vol} +\int_{\partial \Omega} h\,d\mathrm{vol}_{n-1}}{n \int_\Omega h^{\frac{n}{n-1}}\,d\mathrm{vol}}\right)^{n-1}.
\]
\end{proposition}

\begin{proof}
Let $B>0$. Up to replace $h$ with $\tilde h$, We may assume that
\begin{equation}\label{e:scaling}   
\int_\Omega |\nabla h|\,d\mathrm{vol} + B\,n \int_\Omega h\,d\mathrm{vol} +\int_{\partial \Omega} h\,d\mathrm{vol}_{n-1} = n \int_\Omega h^{\frac{n}{n-1}}\,d\mathrm{vol}.
\end{equation}
Consider the Neumann problem
\[
\begin{cases}
\mathrm{div}(h\,\nabla u)= n\,h^{\frac{n}{n-1}}-|\nabla h|-B\,n\,h,&\textrm{in }\Omega\\
\partial_\nu u = 1,& \text{on }\partial\Omega.
\end{cases}
\]
The assumption \eqref{e:scaling}
ensures that the Neumann problem has a solution, which is unique up to an additive constant. Since $h$ is
smooth, we have $|\nabla h| \in C^{0,1}$ and hence by standard elliptic theory (see for example Theorem 6.31
in \cite{GT}) we conclude $u \in C^{2,\alpha}$
 for any $\alpha\in (0,1)$.
 Define
 \[
 U=\{x\in\Omega\setminus\partial\Omega\,:\,|\nabla u(x)|<1\}
 \]
 and 
 \begin{equation}\label{e:contact}
A_r=\{x\in U\,:\,\forall y\in U,\ ru(y)+\frac{1}{2}d(y,\exp_x(r\nabla u(x)))^2 \ge ru(x) +\frac{1}{2}r^2|\nabla u(x)|^2\} 
 \end{equation}
 Define the $C^{1,\alpha}$
transport map $\Phi_r:\Omega\to M$ by
\[
\Phi_r(x)=\exp_x(r\nabla u(x)).
\]
We have the following
\begin{lemma}[cf. Lemma 2.1 in \cite{brendle}]\label{lem:2.1}
Let $x\in U$. Then
\[
\Delta u(x)\le n h(x)^{\frac{1}{n-1}}-B\,n.
\]
\end{lemma}
\begin{proof}
\[
\Delta u= n h^{\frac{1}{n-1}} - \frac{|\nabla h|}{h}-\frac{\langle\nabla h,\nabla u\rangle}{h}- B\,n\le n h^{\frac{1}{n-1}} - B\,n.
\]
\end{proof}

Let $x_\ast$ be the point realizing $\min_{x\in \Omega} u(x)$. Because of the Neumann condition, $x_\ast \not\in \partial\Omega$.

A straightforward improvement of the proof of  \cite[Lemma 2.2]{brendle}, elaborating on a remark from \cite[Lemma 4.4]{IPRV}, gives the following
\begin{lemma}\label{lem:image}
For all $r>0$
\[
B_r(x_\ast)\subset \Phi_r(A_r).\]
\end{lemma}

\begin{proof}
    Fix a point $p\in B_r(x_\ast)$. Consider the function $x\mapsto I_p(x):=r\,u(x) + \frac{1}{2} d(x,p)^2$ defined on $\Omega$ and let $\bar x$ be a point where $I_p$ attains the minimum. Since $u(x)\ge u(x_\ast)$ for all $x\in \Omega$, necessarily $d(\bar x,p) \le d(x_\ast , p) < r$. Hence, $\bar x$ cannot lie on $\partial \Omega$, because $\partial_\nu I_p (x) = r \partial_\nu u(x) + d(x,p) \partial_\nu d(x,p)>0$ on $\partial \Omega \cap B_r(p)$. From now on, the proof is the same as in \cite{brendle}. We report it for completeness.
    Let $\bar \gamma:[0,r] \to  M$ be a minimizing geodesic such that $\bar\gamma(0)=\bar x$ and $\bar\gamma(r)=p$. Clearly, $r|\bar\gamma'(0)|=d(\bar x,p)$. It turns out that $\bar \gamma$ minimizes the functional $ r\,u(\gamma(0)) + \frac 12 r \int_0^r |\gamma'(t)|^2\,dt $ among all smooth path $\gamma:[0,r]\to M$ satisfying $\gamma(0)\in \Omega$ and $\gamma(r)=p$. Indeed, for any such $\gamma$ it holds
    \begin{align*}
        r\,u(\gamma(0)) + \frac 12 r \int_0^r |\gamma'(t)|^2\,dt 
        &\ge  r\,u(\gamma(0)) + \frac{1}{2} d(\gamma(0),p)^2\\
        &\ge r\, u(\bar x)+ \frac{1}{2} d(\bar x,p)^2\\ 
        &= r\, u(\bar\gamma(0))+ \frac{1}{2} r^2 |\bar \gamma'(0)|^2\\ 
        &=  r\,u(\bar\gamma(0)) + \frac 12 r \int_0^r |\bar\gamma'(t)|^2\,dt.
            \end{align*}
        Hence, the formula for the first variation of energy (see for instance \cite[Lemma 5.4.2]{Petersen}) implies $\nabla u(\bar x)=\bar \gamma'(0)$.
        Accordingly, 
        \[
        \Phi_r(\bar x) = \exp_{\bar x}(r\nabla u(\bar x))= \exp_{\bar\gamma(0)}(r\bar \gamma ' (0))= \bar\gamma(r)=p.
        \]
        Moreover, 
\[
r\,|\nabla u(\bar x)| = r\, |\bar\gamma'(0)|= d(\bar x,p) <r,
\]
that is $|\nabla u(\bar x)|<1$ and $\bar x\in U$. Finally, for each $y\in \Omega$, we have
\begin{align*}
    r\,u(y) + \frac{1}{2} d(y,\exp_{\bar x}(r\,\nabla u(\bar x)))^2 
    &= r\,u(y) + \frac 12 d(y,p)^2\\
    &\ge r\,u(\bar x) + \frac 12 d(\bar x,p)^2\\
        &= r\,u(\bar x) + \frac 12 r^2 |\nabla u(\bar x)|^2,
\end{align*}
so that $\bar x\in\ A_r$.
\end{proof}

We recall 
\begin{lemma}[Lemma 2.4 in \cite{brendle}]
Assume that $x \in A_r$ , and let $\bar \gamma(t)= \exp_{\bar x}(t\,\nabla u(\bar x))$ for all $t\in [0,r]$.
Moreover, let $\{e_1,\dots,e_n\}$  be an orthonormal basis of $T_{\bar x}M$. Suppose that $W$
is a Jacobi field along $\bar \gamma$  satisfying $\langle D_tW(0),e_j\rangle = (D^2 u) (W(0),e_j)$ for each
$1\le j\le n$. If $W(\tau)= 0$ for some $\tau\in (0, r)$, then $W$ vanishes identically.
\end{lemma}

The following result is \cite[Proposition 2.5]{brendle}, up to the additional term containing $B$.
\begin{proposition}\label{p:2.5}
Assume that $\mathrm{Ric} \ge 0$ on $M$. Let $\bar x\in A_r$. Then 
\[
|\mathrm{det}D\Phi_r|(\bar x)\le
\left[ 1 +r\,h(\bar x)^{\frac{1}{n-1}}-B\,r\right]^n.
\]
\end{proposition}
A similar procedure leads to 
\begin{proposition}\label{p:2.5 - hyper}
Assume that $\mathrm{Ric} \ge -\kappa ^2$ on $M$. Let $\bar x\in A_r$. Then 
\begin{align}\label{e:Dphi}
|\mathrm{det} D\Phi_r|(\bar x) \le \left\{\cosh\left(\frac{\kappa\, r}{\sqrt{n}}\right) +   \frac{\sqrt{n}\,  \left(h(\bar x)^{\frac{1}{n-1}}-B\right)}{\kappa} \sinh\left(\frac{\kappa\, r}{\sqrt{n}}\right)\right\}^n. 
\end{align}
\end{proposition}

\begin{proof}[Proof (of Proposition \ref{p:2.5} and Proposition \ref{p:2.5 - hyper})]
Choose an orthonormal basis $\{e_1, \dots , e_n\}$
of the tangent space $T_{\bar x}M$, and construct geodesic normal coordinates $(x^1, \dots , x^n)$ around $\bar x$, such
that we have $\partial_i = e_i$ at $\bar x$. Let $\bar \gamma (t)=\exp_{\bar x}(t\,\nabla u(\bar x))$ for all $t\in [0,r]$.
We construct for $1 \le i \le n$ vector fields $E_i$ along $\bar\gamma$
 by parallel transport of the vector fields $e_i$.
Moreover, we solve the Jacobi equation to obtain the unique Jacobi fields $X_i$ along $\bar\gamma$
 satisfying
$X_i(0) = e_i$ and
\[\langle D_tX_i(0), e_j\rangle = (\nabla^2u)(\bar x)(e_i, e_j),\]
where $D_t$ denotes the covariant derivative along $\bar\gamma$.

Let us define a matrix-valued function $P : [0, r ] \to \mathbb R^{n\times n}$ by
\[
[P(t)]_{ij} = \langle X_i(t),E_j (t)\rangle.
\]
In particular, $P(0)=I_n$ and $P'(0)=\nabla^2 u(\bar{x})$ and $\det P(r)=|\mathrm{det} D\Phi_r|(\bar x)$.
Define 
\[
Q(t)=P(t)^{-1}P'(t).
\]
According to Jacobi's formula, $\operatorname{tr} Q(t)=\frac{d}{dt}\log \det P(t)$. Moreover, computing as in \cite{brendle} we obtain the differential relation
\begin{equation}\label{e:trQ}
\frac{d}{dt}\operatorname{tr} Q(t)= - \mathrm{Ric}(\bar\gamma'(t),\bar\gamma'(t))-\operatorname{tr}[Q(t)^2] \le \kappa^2 |\bar\gamma'(t)|^2 - \frac{1}{n}(\operatorname{tr} Q(t))^2,
\end{equation}
where we used the curvature assumption and the Cauchy-Schwarz inequality. Set 
\[
y(t)=[\det P(t)]^{\frac{1}{n}}.
\]
 Then
\begin{align}\label{e:y'}
y'(t)&=\frac{1}{n}y(t)\frac{d}{dt}\log [\det P(t)]\\
&=\frac{1}{n}y(t)[\operatorname{tr} Q(t)],\nonumber
\end{align}
which implies
\begin{equation}\label{e:trQlhs}
\frac{d}{dt}[\operatorname{tr} Q(t)]= n\frac{y''(t)}{y(t)}-n\left(\frac{y'(t)}{y(t)}\right)^2.
\end{equation}
Moreover,
\begin{equation}\label{e:trQrhs}
- \frac{1}{n}(\operatorname{tr} Q(t))^2 = -n\left(\frac{y'(t)}{y(t)}\right)^2 .
\end{equation}
Inserting \eqref{e:trQlhs} and \eqref{e:trQrhs} into \eqref{e:trQ} and using $\bar x\in A_r$ gives 
\begin{equation*}
    n\frac{y''(t)}{y(t)} \le \kappa ^2 |\bar\gamma'(t)|^2= \kappa^2 |\nabla u(\bar x)|^2 \le \kappa^2,
\end{equation*}
i.e.
\[
y''(t)\le \frac{\kappa^2}{n} \,y(t).
\]
Note that $y(0)=1$ and, as a consequence of \eqref{e:y'},
\begin{align*}
    y'(0)&=\frac{1}{n}\operatorname{tr}P^{-1}(0)P'(0) = \frac{1}{n}\Delta u(\bar x).
\end{align*}
A standard ODE comparison then gives
\[
y(t)\le \begin{cases}
1+ \frac{\Delta u(\bar x)}{n} r,& \kappa =0,\\
\cosh\left(\frac{\kappa}{\sqrt{n}}\, t\right) +   \frac{1}{n}\frac{\sqrt{n}}{\kappa }\Delta u(\bar x) \sinh\left(\frac{\kappa}{\sqrt{n}}\, t\right),& \kappa >0.\end{cases}
\]
Hence, using also Lemma \ref{lem:2.1},
\begin{align*}
|\mathrm{det} D\Phi_r|(\bar x) &= \det P(r) = y(r)^n\\
&\le \left[\frac{1}{r}+ h^{\frac{1}{n-1}}(\bar x) -B\right]^n\,r^n \nonumber
\end{align*}
if $\kappa =0$ and 
\begin{align*}
|\mathrm{det} D\Phi_r|(\bar x) \le \left[\cosh\left(\frac{\kappa\,r}{\sqrt{n}}\right) +   \frac{\sqrt{n}\,  \left(h(\bar x)^{\frac{1}{n-1}}-B\right)}{\kappa} \sinh\left(\frac{\kappa\, r}{\sqrt{n}}\right)\right]^n \nonumber
\end{align*}
if $\kappa >0$.
\end{proof}

By Propositions \ref{p:2.5} and \ref{p:2.5 - hyper} and the area formula, 
\begin{align*}
\mathrm{vol}(B_{r}(x_\ast)) &\le  \int_{\Phi_r(A_r)} d\mathrm{vol}\le \int_{A_r} |\det D\Phi_r| \,d\mathrm{vol}
\\
&\le \begin{cases}\int_\Omega \left[ r^{-1} +h^{\frac{1}{n-1}}-B\right]_{+}^n r^n d\mathrm{vol},&\text{if }\kappa=0,\\
\int_\Omega \left[\cosh\left(\frac{\kappa\, r}{\sqrt{n}}\right) r^{-1} +   \frac{\sqrt{n}\,  \left(h^{\frac{1}{n-1}}-B\right)}{\kappa\,r} \sinh\left(\frac{\kappa\, r}{\sqrt{n}}\right)\right]_{+}^n r^n\,  d\mathrm{vol},&\text{if }\kappa>0,
\end{cases}
\end{align*}
so that, by assumption,
\begin{align}\label{e:general}
 \uv(r)\le \begin{cases}\int_\Omega \left[ r^{-1} +h^{\frac{1}{n-1}}-B\right]_{+}^n r^n d\mathrm{vol},&\text{if }\kappa=0,\\
\int_\Omega \left[\cosh\left(\frac{\kappa\, r}{\sqrt{n}}\right) r^{-1} +   \frac{\sqrt{n}\,  \left(h^{\frac{1}{n-1}}-B\right)}{\kappa\, r} \sinh\left(\frac{\kappa\, r}{\sqrt{n}}\right)\right]_{+}^n r^n\,  d\mathrm{vol},&\text{if }\kappa>0.
\end{cases}
\end{align}
\end{proof}

\section{Sobolev-type inequalities}
\label{sec:Sob}

Now, we start making choices which lead to different results.

\subsection{Intermediate volume growth} If $\kappa=0$, Proposition \ref{prop:general} applied with $B=r_0^{-1}$ gives for all $h$ vanishing on $\partial\Omega$
\begin{equation}\label{e:main est0}
\uv(r_0) \le r_0^n \frac{\left(\int_\Omega |\nabla h|\,d\mathrm{vol}+r_0^{-1}\,n\int_\Omega h \,d\mathrm{vol}\right)^n}{\left(\int_\Omega h^{\frac{n}{n-1}}\,d\mathrm{vol}\right)^{n-1}},
\end{equation}
from which 
\begin{equation}\label{e:main est-C}
\uv(r_0)^{\frac{1}{n}} \le r_0 \frac{\int_\Omega |\nabla h|\,d\mathrm{vol}}{\left(\int_\Omega h^{\frac{n}{n-1}}\,d\mathrm{vol}\right)^{\frac{n-1}{n}}}
+
\frac{n\int_\Omega h \,d\mathrm{vol}}{\left(\int_\Omega h^{\frac{n}{n-1}}\,d\mathrm{vol}\right)^{\frac{n-1}{n}}}.
\end{equation}
Note, at this point, that if $\Omega$ is not connected the inequality \eqref{e:main est-C} can be deduced from the same inequality applied to each connected component of $\Omega$, so that from here on we can remove the connectedness assumption.
For $C>1$, choose $r_0$ such that $\uv(r_0)= \frac{\left(Cn\int_\Omega h \,d\mathrm{vol}\right)^n}{\left(\int_\Omega h^{\frac{n}{n-1}}\,d\mathrm{vol}\right)^{n-1}}$. Then, 
\begin{equation*}
\frac{(C-1)n\int_\Omega h \,d\mathrm{vol}}{\left(\int_\Omega h^{\frac{n}{n-1}}\,d\mathrm{vol}\right)^{\frac{n-1}{n}}}
 \le \uv^{-1}\left(\frac{\left(Cn\int_\Omega h \,d\mathrm{vol}\right)^n}{\left(\int_\Omega h^{\frac{n}{n-1}}\,d\mathrm{vol}\right)^{n-1}}\right) \frac{\int_\Omega |\nabla h|\,d\mathrm{vol}}{\left(\int_\Omega h^{\frac{n}{n-1}}\,d\mathrm{vol}\right)^{\frac{n-1}{n}}}.
\end{equation*}
and
\begin{equation}\label{e:main est SC}
(C-1)n\int_\Omega h \,d\mathrm{vol} \le \uv^{-1}\left(\frac{\left(Cn\int_\Omega h \,d\mathrm{vol}\right)^n}{\left(\int_\Omega h^{\frac{n}{n-1}}\,d\mathrm{vol}\right)^{n-1}}\right) \int_\Omega |\nabla h|\,d\mathrm{vol}
\end{equation}

Applying it to $h=h_j\in C^\infty_c(\Omega)$, a sequence point-wisely non-decreasingly converging to $\chi_\Omega$, gives 
\begin{equation}\label{e:C-SC}
(C-1)n \le\uv^{-1}\left(C^n\,n^n\,|\Omega|\right)\frac{|\partial\Omega|}{|\Omega|}.
    \end{equation}
We have thus proved the following result, which recovers \cite[Theorem 3]{CS-C} under stronger assumptions but with different explicit constants.

\begin{theorem}\label{thm:CSC}
    Let $(M^n,g)$ be an $n$-dimensional smooth complete Riemannian manifold without boundary. Suppose that $\mathrm{Ric}_g\ge 0$ and that
    \[
    \uv(r):=\inf_{x\in M} |B_r(x)|>0,\qquad \forall\,r>0.
    \]
Then, for every $C>1$ the
isoperimetric profile $\mathrm I (t)$ of $M$ satisfies  
\[
\mathrm I (t) := \inf\{|\partial \Omega|\ :\ | \Omega|=t\} \ge (C-1)n
\frac{t}{\uv^{-1}(C^n\,n^n\,t)}.\]
\end{theorem}

In a slightly different direction, applying H\"older inequality to the argument of $\uv^{-1}$ in \eqref{e:main est SC}, the latter yields
\begin{equation*}
(C-1)n\int_\Omega h \,d\mathrm{vol} \le \uv^{-1}\left(C^nn^n|\Omega|\right) \int_\Omega |\nabla h|\,d\mathrm{vol}.
\end{equation*}
Using the terminology introduced in \cite{CGL}, we have thus proved that $M$ satisfies a $(1,\psi)$-isoperimetric inequality with $\psi(v)=\frac{n(C-1)}{\uv^{-1}(C^nn^n v)}$.
According to \cite[Proposition 2.1]{CGL}, this implies the validity of a Sobolev type inequality
\[
\int_M |h| F\left(\frac{|h|}{\|h\|_1}\right)\,d\mathrm{vol}\le \int_{M} |\nabla h|\,d\mathrm{vol}
\]
with $F(r)=c(p,\eta)\psi(\eta/r)$. Assuming by rescaling that $\|h\|_1=1$, we have
\[
c(p,\eta)\int_M |h| \frac{n(C-1)}{\uv^{-1}\left(\frac{C^nn^n \eta}{|h|}\right)}\,d\mathrm{vol}\le \int_{M} |\nabla h|\,d\mathrm{vol}.
\]
Whenever $\uv(r)\ge\min\{r^{n},r^{q}\}$, $1\le q\le n$, we get
\[
C(n,q) \int_M |h| \min\left\{\left(\frac{|h|}{\|h\|_1}\right)^{\frac{1}{n}}, \left(\frac{|h|}{\|h\|_1}\right)^{\frac{1}{q}}\right\} \,d\mathrm{vol}\le \int_{M} |\nabla h|\,d\mathrm{vol}.
\]
\begin{remark}
In the EVG case this latter reads
\[
\int_M \frac{|h|^{\frac{n+1}{n}}}{\|h\|_1^{\frac{1}{n}}} \,d\mathrm{vol}\le C(n)^{-1} \int_{M} |\nabla h|\,d\mathrm{vol}.
\]
Such a $L^1$-Nash type inequality is well-known in the literature. It is easy to see that it is implied by the usual Euclidean $L^1$-Sobolev inequality through an H\"older inequality. However, it can be shown to be indeed equivalent; \cite[Theorem 4.3]{BCLS-C} (see also \cite[Proposition 2.2]{CGL}). It is not clear how to extend such an implication to the setting of intermediate volume growth $\uv(r)\ge\min\{r^{n};r^{q}\}$ considered above.
\end{remark}

\subsection{$L^1$-Sobolev inequalities with lower-order terms}\label{sec_varopoulos}
Applying once again Proposition \ref{prop:general} with $\kappa=0$ and $B=r_0^{-1}$ 
 gives
\begin{align*}
\left(\int_\Omega \ h^{\frac{n}{n-1}}d\mathrm{vol}\right)^{\frac{n-1}{n}}
&\le \frac{r_0}{\uv(r_0)^{\frac{1}{n}}}  \int_\Omega \ h^{\frac{n}{n-1}}\,d\mathrm{vol} \\
&=\frac{r_0}{n\,\uv(r_0)^{\frac{1}{n}}}\left(\int_\Omega |\nabla h|\,d\mathrm{vol} + n\,r_0^{-1} \int_\Omega  h\,d\mathrm{vol} + \int_{\partial \Omega} h \,d\mathrm{vol}_{n-1}\right),
\end{align*}
first on each connected component of $\Omega$, and thus on the whole $\Omega$. This prove the first assertion in the statement of Theorem \ref{thm:optimal Sob}. Note that the corresponding isoperimetric type inequality is 
\begin{align*}
|\Omega|^{\frac{n-1}{n}}&\le \frac{r_0}{n\,\uv(r_0)^{\frac{1}{n}}}\left(|\partial \Omega| + n\,r_0^{-1} |\Omega|\right).
\end{align*}
The constant in front of $\int |\nabla h|$ and $|\partial\Omega|$ cannot be improved. Indeed, for $\alpha <1$, consider the singular metric $g=dt^2+\alpha^2t^2g_{\mathbb S^{n-1}}$ on $X=[0,+\infty)\times\mathbb S^{n-1}$. This is the Riemannian structure arising from the cone metric 
\[
d_c((t,\theta),(s,\nu))=\sqrt{t^2+s^2-2ts\cos(\alpha\,d_{\mathbb S^{n-1}}(\theta,\nu))}
\]
on $X$. Note that 
\[
\uv(r_0)= |B_{r_0}(0)|=\int_0^{r_0} \alpha^{n-1} t^{n-1}|\mathbb S^{n-1}|\,dt= \alpha^{n-1}r_0^{n}|\mathbb B_1(0)|.
\]
Indeed, for any $(t,\theta),(\lambda t,\nu)\in X$, with $t>0$, $\lambda>1$ and $\theta,\nu\in \mathbb S^{n-1}$, it holds
\[
|B_{r_0}(t,\theta)|=|B_{r_0}(t,\nu)| = \lambda^{-n} |B_{\lambda r_0}(\lambda t,\nu)| \le |B_{ r_0}(\lambda t,\nu)|, 
\]
where the first equality comes from the rotational symmetry, the second one by the homothetic invariance of the cone, and the last inequality from Bishop-Gromov inequality, as $X$ has nonnegative curvature (e.g. in the Alexandrov sense). 
Choose $\Omega = B_\epsilon(0)$, so that 
\[
|\Omega|^{\frac{n-1}n}= (\alpha^{n-1}\epsilon^{n}|\mathbb B_1(0)|)^{\frac{n-1}{n}}
\]
and
\[
\frac{r_0}{n\,\uv(r_0)^{\frac{1}{n}}}\left(|\partial \Omega| + n\,r_0^{-1} |\Omega|\right) 
= \frac{\alpha^{-\frac{n-1}{n}}}{n\,|\mathbb B_1(0)|^{\frac{1}{n}}} \left(\alpha^{n-1}\epsilon^{n-1}|\partial \mathbb B_1(0)|+ nr_0^{-1}\alpha^{n-1}\epsilon^{n}|\mathbb B_1(0)|\right)
\]
so that, using $|\partial\mathbb B_1(0)|=n|\mathbb B_1(0)|$,
\begin{align}\label{e:sharp}
\frac{r_0}{n\,\uv(r_0)^{\frac{1}{n}}}   \frac{|\partial\Omega| + n\,r_0^{-1} |\Omega|}{|\Omega|^{\frac{n-1}{n}}}=1+O(\epsilon).
\end{align}
Accordingly, up to choosing $\epsilon$ small enough, the ratio in \eqref{e:sharp} is arbitrarily close to one. These singular examples can be smoothed out with an arbitrarily small increment of the ratio in \eqref{e:sharp} while keeping the curvature nonnegative. 
\\\medspace

 If $\kappa >0$, choose $B=\frac{\kappa}{\sqrt{n}}\operatorname{coth}(\frac{\kappa r_0}{\sqrt{n}})$, which yields
\begin{align*}
\left(\int_\Omega \ h^{\frac{n}{n-1}}d\mathrm{vol}\right)^{\frac{n-1}{n}}&\le nA
\int_\Omega \ h^{\frac{n}{n-1}}\,d\mathrm{vol} \\
&= A \left(\int_\Omega |\nabla h|\,d\mathrm{vol} + n\,B \int_\Omega  h\,d\mathrm{vol} + \int_{\partial \Omega} h \,d\mathrm{vol}_{n-1}\right),
\end{align*}
where $A=\frac{\sqrt{n}}{n\kappa\uv(r_{0})^{1/n}}\sinh(\frac{\kappa r_0}{\sqrt{n}})$. We have thus provided an alternative proof based on the ABP method of the classical $L^1$-Sobolev inequality for complete manifolds with lower bounded Ricci curvature $\mathrm{Ric}_g\ge -\kappa^2$ and non-collapsing volumes $\uv(r)>0$, first proved by Varopoulos, \cite{Va,He}.

For $\kappa>0$,
 the constant $A$ in front of $\|\nabla h\|_{L^1}$ is likely not optimal for a fixed $r_0$. However, as $r_0 \to 0$, it is asymptotic to $\frac{r_0}{n\,\uv(r_0)^{1/n}}$. Thus, by the smoothed conical examples discussed above, this constant is sharp in the limit.


\section{Michael-Simon type inequalities under volume noncollapsing}\label{sec:MS}

Throughout this section we will assume that $(M^{n+m},g)$ is a complete smooth $(n+m)$-dimensional Riemannian manifold satisfying
\[
\mathrm{Sect}^M\ge 0.
\]
Let $\Sigma$ be a compact submanifold of dimension $n$. We will denote by $\overline{\nabla}$ the Levi-Civita connection on $M$ and by $\nabla^\Sigma$ the induced connection on $\Sigma$. The second fundamental form $II$ of $\Sigma$ is given by
\[
g( II(X,Y),N)=-g( \overline{\nabla}_X N, Y),
\]
where $X,\, Y$ are tangent vector fields on $\Sigma$ and $N$ is a normal vector field along $\Sigma$. Moreover, the mean curvature vector of $\Sigma$ is defined by $H=\mathrm{tr}II$.\\

\bigskip


 Note that it suffices to prove the result in case $\Sigma$ is connected (we refer e.g. to \cite{brendle} for the discussion of the general case).
 
 Let $h$ be a positive smooth function on $\Sigma$. Let $B>0$. We assume by scaling that
\begin{equation}\label{e:scalingsub}   
\int_\Sigma \sqrt{|\nabla^\Sigma h|^2+h^2|H|^2}\,d\mathrm{vol}_\Sigma + \int_{\partial \Sigma} h \,d\mathrm{vol}_{n-1} + n\,B\int_{\Sigma} h\, d\mathrm{vol}_\Sigma= n \int_\Sigma h^{\frac{n}{n-1}}\,d\mathrm{vol}_\Sigma,
\end{equation}
and we consider the Neumann problem
\[
\begin{cases}
\mathrm{div}^\Sigma(h\nabla^\Sigma u)= nh^{\frac{n}{n-1}}-n \,B\,h -\sqrt{|\nabla^\Sigma h|^2+h^2|H|^2},&\textrm{in }\Sigma\\
\langle \nu,\nabla^\Sigma u\rangle = 1,& \text{on }\partial\Sigma.
\end{cases}
\]
The solution is unique up to an additive constant and continuous on $\bar\Sigma$. 
Let $x_\ast$ be the point realizing $\min_{x\in \Sigma} u(x)$. Then, as a consequence of the Neumann condition, $x_\ast \not\in \partial\Sigma$.

As in \cite{brendle}, define
 \begin{align*}
 \Omega&=\{x\in\Sigma\setminus\partial\Sigma\, :\ |\nabla^{\Sigma} u(x)|<1\}\\
 U&=\{x\in\Sigma\setminus\partial\Sigma\, ,y\in T^\perp_x\Sigma:\,|\nabla^\Sigma u(x)|^2+|y|^2<1\}
 \end{align*}
 and, for $r>0$,
 \begin{equation*}
A_r=\{(\bar{x},\bar{y})\in U\,:\,\forall x\in \Sigma,\ ru(x)+\frac{1}{2}d(x,\exp_{\bar{x}}(r\nabla^\Sigma u(\bar{x})+r\bar{y}))^2 \ge ru(\bar{x}) +\frac{1}{2}r^2(|\nabla^\Sigma u(\bar{x})|^2+|\bar{y}|^2)\} .
 \end{equation*}
Define the $C^{1,\alpha}$
transport map $\Phi_r:T^\perp\Sigma\to M$ by
\[
\Phi_r(x,y)=\exp_x(r\nabla^\Sigma u(x)+ry),
\]
with $\exp$ the exponential map of $M$. 

A minor change to \cite[Lemma 4.1]{brendle} gives
\begin{lemma} Assume that $x\in \Sigma$ and $y\in T_x^{\perp}\Sigma$ satisfy $|\nabla^\Sigma u(x)|^2+|y|^2\le 1$.
Then \[\Delta_\Sigma u(x)- \langle H(x),y\rangle \le n\, h(x)^{\frac{1}{n-1}} - n\,B.\]

\end{lemma}

The main observation is contained in the following lemma, which slightly modifies Lemma 4.2 in \cite{brendle} in the same line as Lemma \ref{lem:image}.

\begin{lemma}\label{Lemma4.2}
The set
$B_{r}(x_\ast)\subset  M^{n+m}$
is contained in $\Phi_r(A_r)$.
\end{lemma}
\begin{remark}
    With respect to \cite[Lemma 4.2]{brendle} this version of the lemma gives a more precise control of the image, notably for small $r$. However, we cannot include in the picture the parameter $\sigma$ used in Brendle's original result to prove the inequality with the sharp constant in the EVG case. 
\end{remark}
\begin{proof}
     Fix a point $p\in B_r(x_\ast)\subset  M^{n+m}$. Consider the function $x\mapsto I_p(x):=r\,u(x) + \frac{1}{2} d(x,p)^2$ defined on $\Sigma$ and let $\bar x\in\Sigma$ be a point where $I_p$ attains the minimum. Since $u(x)\ge u(x_\ast)$ for all $x\in \Sigma$, necessarily $d(\bar x,p) \le d(x_\ast , p) < r$. Hence, $\bar x$ cannot lie on $\partial \Sigma$, because $\partial_\nu I_p (x) = r \partial_\nu u(x) + d(x,p) \partial_\nu d(x,p)>0$ on $\partial \Sigma \cap B_r(p)$. From now on, the proof is the same as the case $\sigma=0$ in \cite[Lemma 4.2]{brendle}; see also the proof of Lemma \ref{lem:image} above.
\end{proof}

We also need the following modified version of \cite[Corollary 4.7]{brendle}. Namely, one has to replace $h^{\frac{1}{n-1}}(\bar x)$ by $h^{\frac{1}{n-1}}(\bar x) - B$ throughout the proof.

\begin{proposition}\label{Prop4.8general}
The Jacobian determinant of $\Phi_r$ satisfies
\[
|\mathrm{det} D\Phi_r|( x, y) \le r^m\left(1+r\,(h( x)^{\frac{1}{n-1}}-B)\right)^n
\]   
for all $(x,y)\in A_r$.
\end{proposition}

With this preparation we can now complete the proof of Theorem \ref{th:Michael-Simon}.
\begin{proof}[Proof of Theorem \ref{th:Michael-Simon}]
Using the coarea formula together with Proposition \ref{Prop4.8general} and computing as in \cite[p.20]{brendle}, we deduce that 
\begin{align*}
|B_r(x_\ast)|&\leq \int_{\{x\in\Sigma: |\nabla^\Sigma u(x)|<1\}} \left(\int_{\{y\in T_x^\perp\Sigma: |\nabla^\Sigma u(x)|^2+|y|^2<1\}}|\mathrm{det}D\Phi_r(x,y)|dy\right)d\mathrm{vol}_\Sigma(x)\\
&\leq \omega_{m}\int_\Sigma  r^{m+n}\left[\frac{1}{r}+h(x)^{\frac{1}{n-1}}-B\right]^n_{+} d\mathrm{vol}_\Sigma(x).
\end{align*}

Choosing $B=r^{-1}$, we obtain
\[
\uv(r)\leq \omega_{m} r^{m+n}\int_\Sigma h^{\frac{n}{n-1}}d\mathrm{vol}_\Sigma(x).
\]
Finally, by the initial rescaling assumption,
\begin{align*}
\frac{n\,\uv(r)^{\frac{1}{n}}}{\omega_{m}^{\frac{1}{n}}r^{1+\frac{m}{n}}}\left(\int_\Sigma h^{\frac{n}{n-1}}\,d\mathrm{vol}_\Sigma\right)^{\frac{n-1}{n}}
&\le n \left(\int_\Sigma h^{\frac{n}{n-1}}\,d\mathrm{vol}_\Sigma\right)^{\frac{1}{n}}\left(\int_\Sigma h^{\frac{n}{n-1}}\,d\mathrm{vol}_\Sigma\right)^{\frac{n-1}{n}}
\\ &= n \int_\Sigma h^{\frac{n}{n-1}}\,d\mathrm{vol}_\Sigma\\
&\le 
\int_\Sigma \sqrt{|\nabla^\Sigma h|^2+h^2|H|^2}\,d\mathrm{vol}_\Sigma + n\,r^{-1}\,\int_\Sigma h\,d\mathrm{vol}_\Sigma+ \int_{\partial \Sigma} h \,d\mathrm{vol}_{n-1}.
\end{align*}
\end{proof}

\begin{remark}
The factor involving $\omega_{m}$ in the constant in Theorem \ref{th:Michael-Simon} comes from the fact that we are using the full normal ball inside the fiber $T_{x}^{\bot}\Sigma$ in the area formula. This is natural under the mere condition of ball-volume non-collapsing, which gives a lower bound on $B_{r}(x_{*})$. In Brendle's setting an additional information is available. Indeed $\mathrm{Sect}^{M}\geq 0$ and $\mathrm{AVR}>0$ imply an asymptotic annular control of the form
\[
\ |B_{r}(x)|-|B_{\sigma r}(x)|\geq \theta\omega_{n+m}(1-\sigma^{n+m})r^{n+m}
\]
as $r\to+\infty$. This is precisely what allows one to introduce the parameter $\sigma$, work with annular regions both in the ambient manifold and in the normal fibers, and let $\sigma\to 1$. Thus, if one assumes a suitably sharp annular noncolapsing condition at a fixed scale, such as
\[
\ |B_{r}(x)|-|B_{\sigma r}(x)|\geq (1-\sigma^{n+m})\underline{v}(r),
\]
then the annular version of the argument recovers Brendle's constant in the limit also in the present setting 
\end{remark}

\section{Geometric estimates for submanifolds with mean curvature}\label{sec:diam}

First, let us point out the following straightforward consequence of Theorem \ref{th:Michael-Simon}.

\begin{corollary}\label{coro:vol}
Let $(M^{n+m},g)$ be a complete smooth $(n+m)$-dimensional Riemannian manifold without boundary satisfying
\[
\mathrm{Sect}^M\ge 0.
\]
Suppose that, for some $r_0>0$, $\uv(r_{0})>0$. Let $\Sigma$ be a closed (i.e. compact without boundary) submanifold of dimension $n$. Then
\[
|\Sigma| \ge \frac{n^n\,\uv(r_0)}{\omega_{m}r_0^{m+n}\,(n\,r_0^{-1}+\|H\|_{\infty})^n}.
\]
\end{corollary}

The following result, which adapts an argument by Topping, \cite{Topping}, to our geometric setting, shows that failure of Euclidean volume growth at a given scale forces concentration of mean curvature. In contrast with Topping’s Euclidean setting (see also \cite{WuZheng}, \cite{Wu} for similar results), this statement holds only up to a fixed scale determined by the Sobolev inequality.

\begin{lemma}\label{delta1}
Let $M^{n+m}$ be a complete Riemannian manifold without boundary with $\mathrm{Sect}_{M}\geq 0$ and $\uv(r):=\inf_{x\in M}|B_{r}(x)|>0$ for some (hence any) $r>0$. Let $\Sigma$ be an $n$-dimensional manifold without boundary isometrically immersed in $M$ which is complete with respect to the induced metric. For $x\in \Sigma$ and $R>0$ define
\begin{align*}
\kappa(x,R):=&\inf_{0<\rho\leq R}\frac{|B^{\Sigma}_{\rho}(x)|}{\rho^{n}},\\
M(x,R):=&\sup_{0<\rho\leq R}\rho^{-\frac{1}{n-1}}|B^{\Sigma}_{\rho}(x)|^{-\frac{n-2}{n-1}}\int_{B_{\rho}(x)}|H|d\mathrm{vol}_{\Sigma}.
\end{align*}
Then, for any fixed $r_{0}>0$,  for every $x\in\Sigma$ and $0<R\leq r_{0}$,
\[
\ \max\left\{\kappa(x,R),M(x,R)\right\}\geq \delta:= c(n,m) \frac{\uv(r_0)}{r_{0}^{m+n}},
\]
for a sufficiently small constant $c(n,m)>0$.
\end{lemma}

\begin{proof}
Choose a scale $r_{0}>0$ and fix $x\in \Sigma$, $0<R\leq r_{0}$. Suppose that
\[
\ M(x,R)< \delta.
\]
Then, for all $r\in(0,R]$
\begin{equation}\label{intHbound}
\int_{B_{r}(x)}|H|d\mathrm{vol}_{\Sigma}\leq \delta r^{\frac{1}{n-1}}|B_{r}^{\Sigma}(x)|^{\frac{n-2}{n-1}}.
\end{equation}
Note that, for fixed $x$, $V(r):=|B_{r}^{\Sigma}(x)|$ is locally Lipschitz, and hence differentiable for a.e. $r$. For such $r$ and any $\mu>0$ define a Lipschitz cutoff on $\Sigma$ by
\[
\ \eta(y)=\begin{cases} 1&y\in B_{r}^{\Sigma}(x)\\1-\frac{1}{\mu}(d_{\Sigma}(x,y)-r)& y\in B_{r+\mu}^{\Sigma}(x)\setminus B_{r}^{\Sigma}(x)\\0&y\notin B_{r+\mu}^{\Sigma}(x) \end{cases}.
\]
Applying \eqref{MS} at scale $r_{0}$ to $h=\eta$ and letting $\mu\downarrow0$ this gives that, for a.e. $r>0$,
\[
\ C(r_{0})V^{\frac{n-1}{n}}\leq V^{\prime}+\int_{B_{r}(x)}|H|d\mathrm{vol}_{\Sigma}+r_{0}^{-1}V,
\]
with $C(r_{0}):=n\omega_{m}^{-1/n}\uv(r_0)^{\frac{1}{n}}r_{0}^{\frac{-m-n}{n}}$. Using \eqref{intHbound}, we obtain
\[
\ V^{\prime}+\delta r^{\frac{1}{n-1}}V^{\frac{n-2}{n-1}}+nr_{0}^{-1}V-C(r_{0})V^{\frac{n-1}{n}}\geq 0.
\]

Suppose now by contradiction that $\kappa(x,R)<\delta$. Note that by Bishop-Gromov theorem, up to choosing $c(n,m)$ sufficiently small, we have that $\delta<\omega_{n}$. Since, as $r\to0$,
\[
\frac{V(r)}{r^{n}}\to\omega_{n}>\delta,
\]
we have that there exists a first radius $s\in (0, R)$ such that $V(s)=\delta s^{n}$, and $V(t)>\delta t^{n}$ for all $0<t<s$. In particular, for every $t\in \left(\frac{s}{2},s\right)$ we have that
\begin{equation*}
V(t)>\delta t^{n},
\end{equation*}
and hence
\begin{equation}\label{BndV1}
\delta t^{\frac{1}{n-1}}V(t)^{\frac{n-2}{n-1}}<\delta^{1-\frac{1}{n(n-1)}}V(t)^{\frac{n-1}{n}}.
\end{equation}
Moreover, since $V$ is nondecreasing, we have that
\begin{equation*}
V(t)\leq V(s)=\delta s^{n}\leq 2^{n}\delta t^{n},
\end{equation*}
and hence
\begin{equation*}
V(t)^{\frac{1}{n}}\leq 2\delta^{\frac{1}{n}}t.
\end{equation*}
Since $t\leq s\leq R\leq r_{0}$, this gives that
\begin{equation}\label{BndV2}
nr_{0}^{-1}V(t)=nr_{0}^{-1}V(t)^{\frac{1}{n}}V(t)^{\frac{n-1}{n}}\leq 2 n\delta^{\frac{1}{n}}V(t)^{\frac{n-1}{n}}.
\end{equation}
Using \eqref{BndV1} and \eqref{BndV2}, we get that for a.e. $t\in(\frac{s}{2},s)$,
\[
V^{\prime}(t)\geq \left(C(r_{0})-\delta^{1-\frac{1}{n(n-1)}}-2 n\delta^{\frac{1}{n}}\right)V(t)^{\frac{n-1}{n}}.
\]
According to the definition of $\delta$ and Bishop-Gromov Theorem, for $c(n,m)$ sufficiently small (which can be chosen independent of $r_0$), we can assume that
\[
\delta^{1-\frac{1}{n(n-1)}}+2 n\delta^{\frac{1}{n}}\leq \frac{C(r_{0})}{2},
\]
and we get
\[
\ V^{\prime}(t)\geq \frac{C(r_{0})}{2}V(t)^{\frac{n-1}{n}},
\]
and therefore
\[
\ \frac{d}{dt}(V(t)^{\frac{1}{n}})\geq \frac{C(r_{0})}{2n}.
\]
Integrating from $\frac{s}{2}$ to $s$ we finally get
\[
\delta^{\frac{1}{n}}s=V(s)^{\frac{1}{n}}\geq \frac{C(r_{0})}{4n}s.
\]
Choosing $c(n,m)$ smaller if necessary, so that 
\[
\ \delta^{\frac{1}{n}}<\frac{C(r_{0})}{4n},
\]
we obtain the desired contradiction.
\end{proof}

As a consequence, we can first obtain a quantitative weak monotonicity statement: small scale-invariant $L^{n-1}$ mean curvature on all intrinsic balls centered at a point forces Euclidean lower volume growth up to the same scale.

\begin{corollary}
In the same assumptions of Lemma \ref{delta1}, for any fixed $r_{0}>0$ there exist $\varepsilon_{0},\,\delta>0$ depending only on $n, m, \uv(r_{0}), r_{0}$ such that for every $x\in\Sigma$ and $0<R\leq r_{0}$, if
\begin{equation}\label{CondH_n-1}
\ \sup_{0<\rho\leq R}\rho^{-1}\int_{B_{\rho}(x)}|H|^{n-1}d\mathrm{vol}_{\Sigma}\leq\varepsilon_{0}
\end{equation}
then
\[
\ |B_{\rho}^{\Sigma}(x)|\geq \delta \rho^{n},
\]
for every $0< \rho\leq R$.
\end{corollary}

\begin{proof}
Note that, by H\"older inequality,
\[
\int_{B_{\rho}(x)}|H|d\mathrm{vol}_{\Sigma}\leq |B_{\rho}^{\Sigma}(x)|^{\frac{n-2}{n-1}}\left(\int_{B_{\rho}^{\Sigma}(x)}|H|^{n-1}d\mathrm{vol}_{\Sigma}\right)^{\frac{1}{n-1}}.
\]
Hence, assuming condition \eqref{CondH_n-1} for a suitable $\varepsilon_{0}$, 
\begin{align*}
M(x,R)\leq& \sup_{0<\rho\leq R}\rho^{-\frac{1}{n-1}}\left(\int_{B_{\rho}^{\Sigma}(x)}|H|^{n-1}d\mathrm{vol}_{\Sigma}\right)^{\frac{1}{n-1}}\\
\leq&\sup_{0<\rho\leq R}\rho^{-\frac{1}{n-1}}\left(\varepsilon_{0}\rho\right)^{\frac{1}{n-1}}=\varepsilon_{0}^{\frac{1}{n-1}}=\frac{\delta}{2},
\end{align*}
with $\delta$ given by Lemma \ref{delta1}. We thus get that $\kappa(x,R)\geq \delta$, from  which the thesis follows.
\end{proof}

Using a covering argument in the spirit of \cite{Topping}, we can derive from Lemma \ref{delta1} also an intrinsic diameter estimate which highlights the interplay between volume and mean curvature in controlling intrinsic distances.

\begin{theorem}\label{th:Topping}
In the same assumptions of Lemma \ref{delta1}, let us assume that $\Sigma$ is closed. Then, for any fixed $r_{0}>0$, there exists a constant $C=C(n, m, \uv(r_{0}), r_{0})$ such that the intrinsic diameter of $\Sigma$, its volume and its mean curvature are related by
\begin{equation}\label{diamEst}
\ \mathrm{diam}_{int}(\Sigma)\leq C\left(\frac{|\Sigma|}{r_{0}^{n-1}}+\int_{\Sigma}|H|^{n-1}d\mathrm{vol}_{\Sigma}\right).
\end{equation}
\end{theorem}

\begin{proof}
By Lemma \ref{delta1}, for any fixed $r_{0}>0$, there exists $\delta=\delta(n, m,\uv(r_{0}), r_{0})\in(0,\omega_{n})$ such that, for every $x\in \Sigma$ and $0<R\leq r_{0}$,
\begin{equation}\label{delta_alt}
\max\left\{\kappa(x,R),M(x,R)\right\}\geq\delta.
\end{equation}
Let $d=\mathrm{diam}_{int}(\Sigma)$ and $\gamma:[0,d]\to\Sigma$ be a minimizing geodesic joining two points at distance $d$. For any $s\in\left[0,d\right]$, set $x_{s}:= \gamma(s)$. Setting $R=\frac{r_{0}}{2}$, by \eqref{delta_alt}, at each point $x_{s}$ one of the following cases occurs:
\begin{itemize}
\item[(1)] We have that $\kappa(x_{s},R)\geq \delta$. In this case, setting $r(s)=\frac{r_{0}}{4}$
\[
\ |B_{r(s)}^{\Sigma}(x_{s})|\geq\delta r(s)^{n}.
\]
\item[(2)] We have that $M(x_{s},R)\geq\delta$. Then, by definition of $M$ and H\"older inequality, there exists $r(s)\in(0,\frac{r_{0}}{2}]$ such that
\begin{equation*}
\frac{\delta}{2}\leq r(s)^{-\frac{1}{n-1}}\left(\int_{B_{r(s)}^{\Sigma}(x_{s})}|H|^{n-1}d\mathrm{vol}_{\Sigma}\right)^{\frac{1}{n-1}}.
\end{equation*}
and hence 
\[
\ r(s)\leq C(n)\delta^{1-n}\int_{B_{r(s)}^{\Sigma}(x_{s})}|H|^{n-1}d\mathrm{vol}_{\Sigma}.
\]
\end{itemize}
Then the family $\left\{B_{r(s)}^{\Sigma}(x_{s})\right\}$ covers the support of $\gamma$. As in \cite{Topping}, by a standard Vitali covering argument, there exists a countable subfamily $\left\{B_{r_{i}}^{\Sigma}(x_{i})\right\}_{i\in\mathcal{I}}$ such that the balls $B_{r_{i}}^{\Sigma}(x_{i})$ are pairwise disjoint and $d\leq\sum_{i}5r_{i}$. At each point $x_{i}$ we choose one of the two alternative above and then we split the indices into the two classes
$\mathcal{I}_{1}$ and $\mathcal{I}_{2}$ depending on whether we are in Case 1 or Case 2. We thus have that
\begin{align*}
\sum_{i\in\mathcal{I}_{1}}r_{i}\leq& C(n)\delta^{-1}\frac{|\Sigma|}{r_{0}^{n-1}},\\
\sum_{i\in\mathcal{I}_{2}}r_{i}\leq& C(n)\delta^{1-n}\int_{\Sigma}|H|^{n-1}d\mathrm{vol}_{\Sigma}.
\end{align*}
Combining these we get 
\begin{equation}\label{diamEst0}
d\leq C(n)\left(\delta^{1-n}\int_{\Sigma}|H|^{n-1}d\mathrm{vol}_{\Sigma}+\delta^{-1}\frac{|\Sigma|}{r_{0}^{n-1}}\right).
\end{equation}
Absorbing $\delta$ into the constant gives the desired diameter estimate.
\end{proof}

\begin{remark}
The volume term in the diameter estimate cannot be removed under the present assumptions. 
Indeed, fix $r_0>0$, and consider the flat manifolds
\[
M_L=S^1(L)\times \mathbb{T}^{n+m-1},
\qquad L\geq 2r_0,
\]
where $S^1(L)$ denotes the circle of length $L$, endowed with its flat metric. Then
\[
\mathrm{Sect}_{M_L}\equiv 0.
\]
Moreover, the family is uniformly noncollapsed at scale $r_0$: there exists
$v_0=v_0(n,m,r_0)>0$, independent of $L$, such that
\[
\inf_{x\in M_L}|B_{r_{0}}(x)|\geq v_{0}.
\]

Let
\[
\Sigma_L=S^1(L)\times \mathbb{T}^{n-1}\times\{p\}
\subset S^1(L)\times \mathbb{T}^{n-1}\times \mathbb{T}^m=M_L.
\]
Then $\Sigma_L$ is a closed totally geodesic $n$-dimensional submanifold, hence
\[
H_{\Sigma_{L}}\equiv 0.
\]
On the other hand,
\[
\mathrm{diam}_{\mathrm{int}}(\Sigma_{L})
\geq \mathrm{diam}(S^{1}(L))=\frac{L}{2},
\]
so the intrinsic diameter tends to infinity as $L\to\infty$. Thus no estimate of the form
\[
\mathrm{diam}_{\mathrm{int}}(\Sigma)
\leq C(n,v_0,r_0)\int_{\Sigma} |H|^{n-1}\,d\mathrm{vol}_{\Sigma}
\]
can hold under the present assumptions. The volume term in our estimate \eqref{diamEst} accounts precisely for the possible presence of such long flat directions.  Note that this example does not meet the assumptions of the diameter estimate in \cite{WuZheng}, since the ambient manifold is not Cartan–Hadamard.
\end{remark}

Under the stronger assumption of intermediate ambient volume growth, the previous diameter estimate can be however sharpened in the direction of \cite{Topping} by choosing the scale according to the volume of $\Sigma$. This rules out the volume alternative in Lemma \ref{delta1} and yields a diameter estimate involving only the curvature term, at the price of a volume-dependent constant.

\begin{theorem}\label{th:ToppingIntermediate}
In the assumptions of Lemma \ref{delta1}, assume that  $\uv(r)\ge \tilde\alpha r^{m+n-q}$ for all $r>1$ and for some $\tilde\alpha>0$ and $q\in [0,n)$. Assume moreover that
\begin{equation}\label{volgrowth}
\sup_{x\in\Sigma}|B_{\rho}^{\Sigma}(x)| \le \beta {\rho^{n-q-p}},\qquad \forall\,\rho>1
\end{equation}
for some $\beta\ge \tilde\alpha\,c(n,m)$, $p>0$, where $c(n,m)$ is the constant in Lemma \ref{delta1}, and assume that $|H|\in L^{n-1}(\Sigma)$.
Then $\Sigma$ has finite intrinsic diameter and hence is compact. Moreover, the intrinsic diameter of $\Sigma$ and its mean curvature are related by
\begin{equation*}\label{diamEst1}
\ \mathrm{diam}_{int}(\Sigma)\leq  c(n,m,\tilde\alpha,\beta,p)\int_{\Sigma}|H|^{n-1}d\mathrm{vol}_{\Sigma}.
\end{equation*}
In particular, if \eqref{volgrowth} is replaced by the stronger assumption that $\Sigma$ has finite volume, then
\[\mathrm{diam}_{int}(\Sigma)\leq  c(n,m,\tilde\alpha)\max\{1\,;\,\mathrm{vol}(\Sigma)^{\frac{(n-1)q}{n-q}}\}\int_{\Sigma}|H|^{n-1}d\mathrm{vol}_{\Sigma}.\]
\end{theorem}\begin{proof}
By Lemma \ref{delta1}, for any fixed $r_{0}>0$ and $x\in \Sigma$ 
\begin{equation*}\label{delta_1}
\max\left\{\kappa(x,r_0),M(x,r_0)\right\}\geq \delta:= c(n,m) \frac{\uv(r_0)}{r_{0}^{m+n}}
\end{equation*}
On the other hand
\begin{align*}
\kappa(x,r_0):=&\inf_{0<\rho\leq r_0}\frac{|B_{\rho}^{\Sigma}(x)|}{\rho^{n}}\le \frac{|B_{r_0}^{\Sigma}(x)|}{r_0^{n}} ,
\end{align*}
If we assume that there are constants  $\alpha\ge c(n,m)$, $p>0$ such that
\[
\sup_{x\in\Sigma}|B_{\rho}^{\Sigma}(x)| \le \alpha \frac{\uv(\rho)}{\rho^{m+p}},\qquad \forall\,\rho>1
\]
then up to choose $r_0> \alpha^{1/p}c(n,m)^{-1/p}$ we get that $\kappa(x,r_{0})<\delta$ and hence, by Lemma \ref{delta1}, that $M(x,r_0)\ge \delta$.
Let $a,b\in\Sigma$ and set $d=d_{\Sigma}(a,b)$. Since $\Sigma$ is complete, there exists a minimizing geodesic $\gamma:[0,d]\to\Sigma$ joining $a$ to $b$. For any $s\in\left[0,d\right]$, set $x_{s}:= \gamma(s)$. At each point $x_{s}$ we have that $M(x_{s},r_0)\geq\delta$. Then, by definition of $M$ and H\"older inequality, there exists $r(s)\in(0,r_{0}]$ such that
\begin{equation*}
\frac{\delta}{2}\leq r(s)^{-\frac{1}{n-1}}\left(\int_{B_{r(s)}^{\Sigma}(x_{s})}|H|^{n-1}d\mathrm{vol}_{\Sigma}\right)^{\frac{1}{n-1}}.
\end{equation*}
and hence 
\[
\ r(s)\leq 2
^{n-1}\delta^{1-n}\int_{B_{r(s)}^{\Sigma}(x_{s})}|H|^{n-1}d\mathrm{vol}_{\Sigma}.
\]
Then the family $\left\{B_{r(s)}^{\Sigma}(x_{s})\right\}$ covers the support of $\gamma$. As in \cite{Topping}, by a standard Vitali covering argument, there exists a countable subfamily $\left\{B_{r_{i}}^{\Sigma}(x_{i})\right\}_{i\in\mathcal{I}}$ such that the balls $\left\{B_{r_{i}}^{\Sigma}(x_{i})\right\}$ are pairwise disjoint and $d\leq\sum_{i}5r_{i}$.  We thus have that
\begin{align*}
d&\le5\sum_{i\in\mathcal{I}}r_{i}\\
&\leq 2^{n-1}5\,\delta^{1-n}\int_{\Sigma}|H|^{n-1}d\mathrm{vol}_{\Sigma}.
\end{align*}
Since $a,b\in\Sigma$ are arbitrary, taking the supremum over $a,b$, gives 
\[
\ \mathrm{diam}_{int}(\Sigma)\leq 2^{n-1}5\,\delta^{1-n}\int_{\Sigma}|H|^{n-1}d\mathrm{vol}_{\Sigma}.
\]
In particular, $\Sigma$ has finite intrinsic diameter. Since $\Sigma$ is complete, Hopf-Rinow theorem implies that $\Sigma$ is compact, hence closed. Here $\delta$ depends explicitly on $\alpha, p, n,m$. Namely,
\[
\delta = c(n,m)\frac{\uv(\alpha^{1/p}c(n,m)^{-1/p})}{(\alpha^{1/p}c(n,m)^{-1/p})^{m+n}}.
\]
In particular if $\uv(\rho)=\tilde\alpha\rho^{m+n-q}$ when $\rho>1$ for some $q<n$, then the result works as soon as
\[
\sup_{x\in\Sigma}|B_{\rho}^{\Sigma}(x)| \le \alpha \tilde\alpha {\rho^{n-q-p}},\qquad \forall\,\rho>1
\]
for some $\alpha\ge c(n,m)$, $p>0$, and in this case
\[
\delta = c(n,m)\tilde\alpha\frac{1}{(\alpha^{1/p}c(n,m)^{-1/p})^{q}}=c(n,m)^{1+q/p}\frac{\tilde \alpha}{\alpha^{q/p}}
\]
In case $\Sigma$ has finite volume, one can choose $p=n-q$ and $\alpha\ge \max\{\tilde\alpha^{-1}|\Sigma|\,;\,c(n,m)\}$
so that
$\delta\ge c(n,m,\tilde\alpha)\max\{1\,;\,|\Sigma|^{\frac{-q}{n-q}}\}.$
\end{proof}

\begin{remark}
In the special case where $M$ has Euclidean volume growth $\mathrm{AVR}(M)=\theta>0$, (the proof of) Theorem \ref{th:ToppingIntermediate} gives that $\Sigma$ is necessarily closed as soon as 
$\sup_{x\in\Sigma}|B_\rho^{\Sigma}(x)| = o(\rho^n)$.
This can be interpreted as a sort of dichotomy for volume growth behaviours. Namely, if $H=0$ then the volumes of the intrinsic geodesic balls of $\Sigma$ are at least Euclidean because of the sharp isoperimetric inequality. If $H$ is nonzero but $|H|\in L^{n-1}$, then either $\Sigma$ is closed, or the volume growth has to be almost Euclidean.
This fact does not seem to have been explicitly emphasized in the literature even for submanifolds of $\mathbb R^n$, although in that case it is a straightforward consequence of the proof of Topping's theorem. 
\end{remark}

\section{Isoperimetric and analytic consequences for minimal submanifolds}\label{Sec: ApplMinImm}

Suppose in this section that $\Sigma$ is an $n$-dimensional smooth minimal submanifold with smooth (possibly empty) boundary immersed in $M^{n+m}$. Suppose also that
$\frac{\uv(r)}{r^m}\ge \mathcal V(r)$ for some positive increasing function $\mathcal V$ which diverges at infinity. 
We have the following series of consequences
\begin{proposition}\label{pr:Grig}
Suppose that $M^{n+m}$ is a complete Riemannian manifold without boundary with $\sect_M\ge 0$. Let $\uv(r):=\inf_{x\in M}|B_r(x)|$ and suppose that
$\frac{\uv(r)}{r^m}\ge \mathcal V(r)$ for some positive increasing  function $\mathcal V$ which diverges at infinity.
\begin{enumerate}
    \item 
  Let $\Sigma$ be a compact $n$-dimensional minimal submanifold with (possibly empty) smooth boundary immersed in $M$. Then
\begin{enumerate}
    \item (isoperimetric inequality) For all domains $D\subset \Sigma$ with smooth boundary 
    \[
|\partial D| \ge \frac{n|D|}{\mathcal V^{-1}\left(2^{n}\omega_{m}|D|\right)}.
\]
\item (spectral gap) For all domains $D\subset \Sigma$ with smooth boundary the Dirichlet eigenvalue is controlled by
    \[
\lambda_1( D) \ge \frac{n^2}{4\,\left(\mathcal V^{-1}\left(2^{n}\omega_{m}|D|\right)\right)^2}.
\]
\end{enumerate}
\item
  Let $\Sigma$ be a complete $n$-dimensional minimal submanifold without boundary immersed in $M$. Then
\begin{enumerate}
\item (off-diagonal Gaussian estimates) Define a function $\zeta$ by the relation
\begin{equation}\label{e:def zeta}
t= \int_0^{\zeta(t)}\frac{4\,\left(\mathcal V^{-1}(2^{n}\omega_{m}v)\right)^2}{n^2v}\,dv.
\end{equation}
If for some
$T \in (0,+\infty]$ the function $\frac{t\zeta'(t)}{\zeta(t)}$ is increasing for $t>T$ and bounded for $t<2T$, then $\forall\,x,y\in \Sigma,\ \forall\,t>0$
\[
h(x,y,t)\le \frac{C_1}{\zeta(C_2t)} \exp\left(-\frac{d(x,y)^2}{C_3\,t}\right)
\]
with $C_1,C_2,C_3$ positive constants, $C_3$ arbitrarily close to $4$.
\item (higher eigenvalues)
If $(\log \zeta)'(t)$ has at most polynomial decay, then for any pre-compact region $D\subset \Sigma$ we have $\lambda_k(D) \ge \frac{C_4}{\left(\mathcal V^{-1}\left(C_52^{n}\omega_{m}\frac{|D|}{k}\right)\right)^2}$ for all $k = 1, 2, \dots$ with $C_4,C_5$ positive constants.
\end{enumerate}
\end{enumerate}
\end{proposition}

\begin{proof}
Let $D\subset \Sigma$. Choosing $h$ a smooth approximation of $\chi_D$ on $\Sigma$ we get
\begin{align*}
\frac{n\mathcal V(r)^{\frac{1}{n}}}{\omega_{m}^{\frac{1}{n}}r}\le\frac{n\,\uv(r)^{\frac{1}{n}}}{\omega_{m}^{\frac{1}{n}}r^{1+\frac{m}{n}}}
&\le 
    \frac{|\partial D|}{|D|^{\frac{n-1}{n}}} + n\,r^{-1}|D|^{\frac{1}{n}}\end{align*}
for all $r>0$. Choosing $r$ such that $\mathcal V(r)=2^n\omega_{m}|D|$ gives
\[
|\partial D| \ge \frac{n|D|}{\mathcal V^{-1}\left(2^{n}\omega_{m}|D|\right)},
\]
which proves (1.a). The remaining statements are straightforward applications of results from \cite{gr-hkub}, namely Proposition 2.4 for (1.b) and Theorem 1.1 for (2). See also \cite{car}.
\end{proof}

\begin{remark}
If $\Sigma$ is a complete submanifold, then point (1) in the previous proposition implies a volume growth estimate. Namely, if
    \[
\frac{\partial}{\partial r}|B_r^{\Sigma}(x)| =|\partial B_r^{\Sigma}(x)| \ge \frac{n|B_r^{\Sigma}(x)|}{\mathcal V^{-1}\left(2^{n}\omega_{m}|B_r^{\Sigma}(x)|\right)}
\]
for any $x\in \Sigma$ and a.e. $r$, by ODE comparison one gets that
$|B_r^{\Sigma}(x)|\ge \xi (r)$ where the function $\xi$ is defined by the relation
\[
t = \int_0^{\xi(t)}\frac{\mathcal V^{-1}(2^{n}\omega_{m}v)}{nv}\,dv
\]
In particular volume growth estimates for $M$ imply intrinsic volume growth estimates for $\Sigma$.
\end{remark}

A special case where the approach above works is when we assume that \[
\uv(r)\ge \tilde\alpha\,r^{m+n-q},\qquad  r\ge 1
\]
for some $\tilde\alpha>0$ and $q\in [0,n)$. By Bishop-Gromov inequality, this in turn implies
\[
\uv(r)\ge \min\left\{\tilde\alpha\,r^{m+n}\,,\,\tilde\alpha\,r^{m+n-q}\right\}.
\]
Let us verify that the assumptions of Proposition \ref{pr:Grig} are satisfied. To this end, let $F$ be a positive $C^1$ function on $(0,+\infty)$ such that $F(s)=n$ in $[0,1]$, $F(s)=(n-q)(1-\frac{1}{s})$ in $[2,\infty)$ and $n\ge F(s)\ge \frac{n-q}{4}$ in $[1,2]$.
Then, 
\[
\frac{\uv(r)}{r^{m}}\ge\min\left\{\tilde\alpha\,r^{n}\,,\,\tilde\alpha\,r^{n-q}\right\}\ge
\frac{\tilde\alpha}{2^q}\exp\int_1^{r}\frac{F(\rho)}{\rho}\,d\rho=:\mathcal V(r).
\]
For later purposes, note also that \[
\mathcal V(r)\ge \frac{\tilde \alpha}{2^q}\left(r^n\chi_{[0,1]}(r)+r^{\frac{n-q}{4}}\chi_{(1,2]}(r)+2^{-\frac{3}{4}(n-q)}e^{-\frac{n-q}{2}}r^{n-q}\right)\ge \alpha'\min\{r^n\,,\,r^{n-q}\}\]
with $\alpha'=\frac{\tilde \alpha}{2^q}2^{-\frac{3}{4}(n-q)}e^{-\frac{n-q}{2}}$, so that
\[
\max\left\{\left(\frac{s}{\tilde\alpha}\right)^{\frac 1n}\,,\,\left(\frac{s}{\tilde\alpha}\right)^{\frac 1{n-q}}\right\}\le\mathcal V^{-1}(s)\le \max\left\{\left(\frac{s}{\alpha'}\right)^{\frac 1n}\,,\,\left(\frac{s}{\alpha'}\right)^{\frac 1{n-q}}\right\}.\]
Note that $\mathcal V'>0$ and $\lim_{r\to\infty} \mathcal V(r)=+\infty$. 
Differentiating \eqref{e:def zeta} we obtain $1= \frac{4\,\mathcal V^{-1}(2^{n}\omega_{m}\zeta(t))^2}{n^{2}\zeta(t)}\zeta'(t)$. Set a new variable $S=\mathcal V^{-1}(2^{n}\omega_{m}\zeta(t))$. Then, $1=\frac{4S^2}{n^{2}\zeta(t)}\zeta'(t)$, i.e. $\zeta'(t)=\frac{n^{2}\zeta(t)}{4S^{2}}=\frac{n^{2}\mathcal V(S)}{4\cdot2^{n}\omega_{m}S^2}$, so that 
\[
\frac{n^{2}t}{4S^2}=\frac{t\zeta'(t)}{\zeta(t)}=:H(t).
\]
Again by \eqref{e:def zeta},
\[t(S) = \int_0^{\mathcal V(S)/(2^{n}\omega_{m})}\frac{4\,(\mathcal V^{-1}(2^{n}\omega_{m}v))^2}{n^{2}v}\,dv = \int_0^S \frac{4s^2}{n^{2}\mathcal V(s)} \mathcal V'(s)\, ds\]
so that in the new variable $S$ it holds
\[
H(S)= \frac{1}{S^2}\int_0^S \frac{s^2}{\mathcal V(s)}\mathcal V'(s)\,ds=\frac{1}{S^2}\int_0^S s\,F(s)\,ds.
\]
As $\zeta$ and $\mathcal V$ are increasing, $\left(\frac{t\zeta'(t)}{\zeta(t)}\right)'>0$ if and only if $0<\frac{d}{dS}H(S)$. We compute
\[
\frac{d}{dS}H(S) = \frac{S^2F(S)-2\int_0^S s\,F(s)\,ds}{S^3}.
\]
Since 
\[
\frac{d}{dS}\left(S^2F(S)-2\int_0^S s\,F(s)\,ds\right)=S^2F'(S) = n-q
\]
for $S>2$, we deduce that $H'(S)>0$ for $S$ large enough, and thus also $\left(\frac{t\zeta'(t)}{\zeta(t)}\right)'>0$ for $t$ large enough. 
Moreover, $H$ is clearly bounded as $F$ is bounded and bounded away from $0$. Accordingly, the regularity assumptions of Proposition \ref{pr:Grig} (2.a) are satisfied.

Similarly, observe that $(\log\zeta(t))'=\frac{n^{2}}{4S^2}$ is monotonically decreasing (as a function of $S$, hence as a function of $t$, $\mathcal V$ being strictly increasing). Moreover
\begin{align*}
t=&  \int_0^{\zeta(t)}\frac{4\,\mathcal V^{-1}(2^{n}\omega_{m}v)^2}{n^{2}v}\,dv
\\ \le& 
\begin{cases}
\int_0^{\zeta(t)}\frac{4}{n^{2}}\,(\frac{2^{n}\omega_{m}}{\alpha'})^{\frac2n}\,v^{\frac{2}{n}-1}\,dv= \frac{2}{n}\,(\frac{2^{n}\omega_{m}}{\alpha'})^{\frac2n}\, \zeta(t)^{\frac{2}{n}}   &if\ \zeta(t) \le \frac{\alpha'}{2^{n}\omega_{m}}\\
\frac{2}{n} + \int_{\frac{\alpha'}{2^{n}\omega_{m}}}^{\zeta(t)} \frac{4}{n^{2}}\left(\frac{2^{n}\omega_{m}}{\alpha^{\prime}}\right)^{\frac{2}{n-q}}\,v^{\frac{2}{n-q}-1}dv=\frac{2}{n}-\frac{2(n-q)}{n^{2}}+ \frac{2(n-q)}{n^{2}}\left(\frac{2^{n}\omega_{m}}{{\alpha'}}\right)^{\frac2{n-q}}\zeta(t)^{\frac{2}{n-q}}\,& if\ \zeta(t) \ge \frac{\alpha'}{2^{n}\omega_{m}},
\end{cases}
\end{align*}
so that
\[
\zeta(t) \ge 
C(n,q,m,\alpha') \min\{ t^{\frac n2}\,,\,
t^{\frac {n-q}2}\}.\]
Similarly,
\[
\zeta(t) \le 
C(n,q,m,\tilde\alpha) \max\{ t^{\frac n2}\,,\,
t^{\frac {n-q}2}\}.\]
Accordingly, $(\log\zeta(t))'$
has polynomial decay in the sense of \cite[Definition 2.1]{gr-hkub}, so that also the regularity assumptions of Proposition \ref{pr:Grig} (2.b) are satisfied.

We have thus proved the following
\begin{corollary}\label{coro:Grig} 
    Suppose that $M^{n+m}$ is a complete Riemannian manifold without boundary with $\sect_M\ge 0$. Suppose that
    \[\uv(r):=\inf_{x\in M}|B_r(x)| \ge \tilde\alpha\,r^{m+n-q},\qquad  \forall\,r\ge 1
\]
for some $\tilde\alpha>0$ and $q\in [0,n-1)$.
\begin{enumerate}
    \item 
  Let $\Sigma$ be a compact $n$-dimensional minimal submanifold with (possibly empty) smooth boundary immersed in $M$. Then, there exists positive constants $c_1,c_2,c_3$ such that
\begin{enumerate}
    \item (isoperimetric inequality) For all domains $D\subset \Sigma$ with smooth boundary 
    \[
|\partial D| 
\ge c_2\min\{|D|^{\frac {n-1}n},|D|^{\frac {n-q-1}{n-q}}\} 
\]
\item (spectral gap) For all domains $D\subset \Sigma$ with smooth boundary the Dirichlet eigenvalue is controlled by
    \[
\lambda_1( D) \ge c_3\min\{|D|^{-\frac {2}n},|D|^{-\frac {2}{n-q}}\} 
\]
\end{enumerate}
\item
  Let $\Sigma$ be a complete $n$-dimensional minimal submanifold without boundary immersed in $M$. Then
\begin{enumerate}
\item (off-diagonal Gaussian estimates)  $\forall\,x,y\in \Sigma,\ \forall\,t>0$
\[
h(x,y,t)\le \frac{C_1}{\min\{t^{\frac{n}{2}},t^{\frac{n-q}{2}}\}} \exp\left(-\frac{d(x,y)^2}{C_3\,t}\right)
\]
with $C_1,C_3$ positive constants, with $C_3$ arbitrarily close to $4$.
\item (higher eigenvalues)
For any pre-compact region $D\subset \Sigma$ we have \[
\lambda_k( D) \ge C_4\min\{k^{\frac 2n}|D|^{-\frac {2}n},k^{\frac{2}{n-q}}|D|^{-\frac {2}{n-q}}\} 
\] for all $k = 1, 2, \dots$ with $C_4$ a positive constant.
\item (volume growth)
For any $x\in \Sigma$ we have
\[
|B_r^\Sigma(x)|\ge C_5 \min\{r^{n},r^{n-q}\},
\]
with $C_5$ a positive constant.
\end{enumerate}
\end{enumerate}
\end{corollary}


\section*{Data availability}
Data sharing not applicable to this article as no datasets were generated or analyzed during
the current study.

\section*{Conflict of interest statement}
On behalf of all authors, the corresponding author states that there is no conflict of interest.

\bibliographystyle{alpha}
\bibliography{ABP}

\end{document}